\documentclass[a4paper]{amsart} 

\usepackage{microtype,graphicx,subfig} 
\usepackage[backref=page]{hyperref}
\usepackage{amssymb, amsrefs, color}
\usepackage[utf8]{inputenc}
\usepackage[english]{babel}

\newtheorem{theorem}{Theorem}[section]

\newtheorem{lemma}[theorem]{Lemma}

\newtheorem{proposition}[theorem]{Proposition}
\newtheorem{problem}[theorem]{Problem}

\newtheorem{remark}[theorem]{Remark}

\theoremstyle{definition}

\theoremstyle{remark}

\newcommand{\R}{\mathbb{R}}
\newcommand{\C}{\mathbb{C}}

\newcommand{\Z}{\mathbb{Z}}
\newcommand{\Pc}{\mathbb{P}(\mathbb{C}^2)}
\renewcommand{\S}{\mathbb{S}^2}

\newcommand{\CH}{\ \mathcal{H}_N}
\newcommand{\dist}{\operatorname{dist}}

\title{A sequence of polynomials with optimal condition number}
\author{Carlos Beltrán, Ujué Etayo, Jordi Marzo and Joaquim Ortega-Cerdà}

\thanks{Carlos Beltrán and Ujué Etayo were partially  supported by Ministerio de 
Economía y Competitividad, Gobierno de España, through grants MTM2017-83816-P,  
MTM2017-90682-REDT, and by the Banco de Santander and Universidad de Cantabria 
grant 21.SI01.64658. Joaquim Ortega-Cerdà and Jordi Marzo have
been partially supported by grant MTM2017-83499-P by the 
Ministerio de Econom\'{\i}a y Competitividad, Gobierno de Espa\~na and by the 
Generalitat de Catalunya (project 2017 SGR 358).}

\begin{document}
\begin{abstract}
We find an explicit sequence of univariate polynomials of arbitrary degree with 
optimal condition number. This solves a problem posed by Michael Shub and 
Stephen Smale in 1993.
\end{abstract}
\maketitle

\section{Introduction}
\subsection{The Weyl norm and the condition number of polynomials}

Closely following the notation of the celebrated paper \cite{BezIII}, we denote 
by $\CH$ the vector space of bivariate homogeneous polynomials of degree $N$, 
that is the set of polynomials of the form
\begin{equation}\label{eq:polyg}
g(x,y)=\sum_{i=0}^Na_i x^iy^{N-i},\quad a_i\in\C
\end{equation}
where $x,y$ are complex variables. The Weyl norm of $g$  (sometimes called 
Kostlan or Bombieri-Weyl or Bombieri norm) is
\[
\|g\|=\left(\sum_{i=0}^N\binom{N}{i}^{-1}|a_i|^2\right)^{1/2},
\]
where the binomial coefficients in this definition are introduced to satisfy 
the property $\|g\|=\|g\circ U\|$ where $U\subseteq\C^{2\times 2}$ is any 
unitary $2\times 2$ matrix and $g\circ U\in\CH$ is the polynomial given by 
$g\circ U(x,y)=g(U\binom{x}{y})$. Indeed, with this metric we have
\[
\|g\|^2=\frac{N+1}{\pi}\int_{ \Pc}\frac{|g(\eta)|^2}{\|\eta\|^{2N}}\,dV(
\eta),
\]
where the integration is made with respect to volume form $V$ 
arising from the standard Riemannian structure in $\Pc$. Note that the 
expression inside the integral is well defined since it does not depend on the 
choice of the representative of $\eta\in\Pc$.

The zeros of $g$ lie naturally in the complex projective space $\Pc$. The 
condition number of $g$ at a zero $\zeta$ is defined as follows. If the 
derivative $Dg(\zeta)$ does not vanish, by the Implicit Function Theorem the 
zero $\zeta$ of $g$ can be continued in a unique differentiable manner to a zero 
$\zeta'$ of any sufficiently close polynomial $g'$. This thus defines (locally) 
a {\em solution map} given by $Sol(g')= \zeta'$. The condition number is by 
definition the operator norm of the derivative of the solution map, in other 
words $\mu(g,\zeta)=\|D Sol(g,\zeta)\|$, where the tangent spaces $T_g\CH$ and 
$T_\zeta\Pc$ are endowed respectively with the Bombieri--Weyl norm and the 
Fubini--Study metric. In \cite{BezI} it was proved that
\begin{equation}\label{def:mu}
\mu(g,\zeta)=\|g\|\,\|\zeta\|^{N-1}|(Dg(\zeta)\mid_{\zeta^{\perp}})^{-1}|,
\end{equation}
(the definition and theory in \cite{BezI} applies to the more general case of 
polynomial systems). Here, $Dg(\zeta)$ is just the derivative
\[
Dg(\zeta)=\left(\frac{\partial}{\partial x} g(x,y) \quad  
\frac{\partial}{\partial y} g(x,y)\right)_{(x,y)=\zeta}
\]
and $Dg(\zeta)\mid_{\zeta^{\perp}}$ is the restriction of this derivative to 
the orthogonal complement of $\zeta$ in $\C^2$. If this restriction is not 
invertible, which corresponds to $\zeta$ being a double root of $g$, then by 
definition $\mu(g,\zeta)=\infty$.

Shub and Smale also introduced a normalized version of the condition number 
since it turns out to produce more beautiful formulas in the later development 
of the theory (very remarkably in the extension to polynomial systems), see for 
example \cite{BlCuShSm} or \cite{Condition}. In the case of polynomials it is 
simply defined by
\begin{equation}\label{def:munorm}
\mu_{\rm 
norm}(g,\zeta)=\sqrt{N}\,\mu(g,\zeta)=\sqrt{N}\,\|g\|\,\|\zeta\|^{N-1}|(Dg(\zeta)\mid_{\zeta^{\perp}})^
{-1}|.
\end{equation}
The normalized condition number of $g$ (without reference to a particular 
zero) is defined by
\[
\mu_{\rm norm}(g)=\max_{\zeta\in \Pc:g(\zeta)=0}\mu_{\rm norm}(g,\zeta).
\]
Now, given a univariate degree $N$ complex polynomial 
$P(z)=\sum_{i=0}^Na_iz^i$, it has a homogeneous counterpart 
$g(x,y)=\sum_{i=0}^Na_i x^iy^{N-i}$. The condition number and the Weyl norm of 
$p$ are defined via its homogenized version:
\[
\|P\|=\|g\|,\quad \mu_{\rm norm}(P,z)=\mu_{\rm norm}(g,(z,1)),
\]
\[
\mu_{\rm norm}(P)=\mu_{\rm norm}(g)=\max_{z\in \C:P(z)=0}\mu_{\rm norm}(P,z).
\]
A simple expression for the condition number of a univariate polynomial (see 
for example \cite{Facility}) is:
\begin{equation}\label{eq:mucomp}
\mu_{\rm norm}(P,z)=N^{1/2}\frac{\|P\| (1+|z|^2)^{N/2-1}}{|P'(z)|},
\end{equation}
and we have $\mu_{\rm norm}(P,z)=\infty$ if and only if $z$ is a double zero of 
$P$. For example, the condition number of the polynomial $z^N-1$ is equal at all 
of its zeros and
\begin{equation}\label{eq:aicesunidad}
\mu_{\rm 
norm}(z^N-1)=N^{1/2}\frac{\|z^N-1\|2^{N/2-1}}{N}=\frac{2^{N/2-1/2}}{\sqrt{N}}.
\end{equation}
(Note that the same computation gives a slightly different result in 
\cite[p. 7]{BezIII}; the correct quantity is \eqref{eq:aicesunidad}).

\subsection{The problem of finding a sequence of well--conditioned polynomials}
In \cite{BezII} it was proved that, if $P$ is uniformly chosen in the unit 
sphere of $\CH$ (i.e. the set of polynomials of unit Weyl norm, endowed with the 
probability measure corresponding to the metric inherited from $\CH$) then 
$\mu_{\rm norm}(P)$ is smaller than $N$ with probability at least $1/2$. Indeed, 
as pointed out in \cite{BezIII}, with positive probability a polynomial of 
degree $N$ with $\mu_{\rm norm}(P)\leq N^{3/4}$ can be found. In other words, 
there exist plenty of degree $N$ polynomials with rather small condition number.

Indeed, the least value that $\mu_{\rm norm}$ can attain for a degree $N$ 
polynomial seems to be unknown. We prove in Section \ref{sec:lowerbound} the 
following lemma.
\begin{lemma}\label{lem:lb}
	There is a universal constant $C$ such that $
	\mu_{\rm norm}(P)\geq C\sqrt{N}$ for every degree $N$ polynomial $P$.
\end{lemma}
Despite the existence of well--conditioned polynomials of all degrees, 
explicitly describing such a sequence of polynomials was proved to be a 
difficult 
task, which lead to the following:
\begin{problem}[Main Problem in \cite{BezIII}]\label{problem:main}
Find explicitly a family of polynomials $P_N$ of degree $N$ with $\mu_{\rm 
norm}(P_N)\leq 
N$.
\end{problem}
By ``find explicitly'' Shub and Smale meant ``giving a handy description'' or 
more formally describing a polynomial time machine in the BSS model of 
computation describing $P_N$ as a function of $N$. Indeed, Shub and Smale 
pointed out that it is already difficult to describe a family such that 
$\mu_{\rm norm}(P_N)\leq N^k$ for any fixed constant $k$, say $k=100$. Despite 
the existence of many well conditioned polynomials, we cannot even find one! 
This fact was recalled by Michael Shub in his plenary talk at the FoCM 2014 
conference where he referred to the problem as {\em finding hay in the 
haystack}.

One of the reasons that lead Shub and Smale to pose the question above was the 
possible impact on the design of efficient algorithms for solving polynomial 
equations. In short, a homotopy method to solve a target polynomial $P_1$ will 
start by choosing another polynomial of the same degree $P_0$ all of whose roots 
are known and will try to follow closely the path of solutions of the polynomial 
segment $P_t=(1-t) P_0+t P_1$. Shub and Smale noticed that if $P_0$ has a large 
condition number then the resulting algorithm will be unstable, thus the 
interest in finding an explicit expression for some well--conditioned sequence. 
The reverse claim (that a well conditioned polynomial will produce efficient and 
stable algorithms) is quite nontrivial, yet true: it was proved in \cite{BuCu} 
that if $P_0$ has a condition number which is bounded by a polynomial in $N$ 
then the total expected complexity of a carefully designed homotopy method is 
polynomial in $N$ for random inputs. The question of finding a good starting 
pair for the homotopy (which is the core of Smale's 17th problem \cite{Smale}) 
has actually been solved by other means even in the polynomial system case, see 
\cite{FLH,BuCu,Lairez} that solve Smale's 17th problem and subsequent papers 
which improve on these results. Yet, Problem \ref{problem:main} remained 
unsolved. It was also included as Problem 12 in \cite[Chpt: Open 
Problems]{Condition}, and there were several unsuccesful attempts to solve it 
via some particular constructions of polynomials that seemed to behave well, but 
only numerical data was produced. 

\subsection{Relation to spherical points and 
Smale's 7th problem}

Given a point $z\in\mathbb C$ we denote by $\hat z$ the point in $\S = 
\{(a,b,c)\in\R^3: \ a^2+b^2+c^2 = 1\}$ 
obtained from the stereographic projection.
That is if we denote $\hat z = (a, b, c)$ then $z = 
(a+ib)/(1-c)$ and conversely 
\[
a = \frac{z+\bar z}{1+|z|^2},\qquad
b =  \frac{z-\bar z}{i(1+|z|^2)},\qquad
c =  \frac{|z|^2-1}{1+|z|^2}.
\]
Given $P(z) = \prod_{i = 1}^N (z-z_i)$ we 
consider the continuous function $\hat P: \S\to \mathbb R$ defined as
$\hat P(x) = \prod_{i = 1}^N |x-\hat z_i|$. Moreover for any given zero 
$\zeta$ of $P$ we define $\hat P_\zeta(x) = \hat P(x)/|x-\hat\zeta|$, 
that in the case $x=\zeta=\hat z_i$ for some $i$ simply means
\[
\hat P_{z_i}(\hat z_i)=\prod_{j\neq i}|\hat z_i-\hat z_j|.
\]
With 
this 
notation, \cite{BezIII}*{Proposition~2} claims that
\begin{equation}\label{def:workingmu}
\mu_{\rm norm}(P, \zeta) = \frac12\sqrt{N(N+1)}\frac{\|\hat P\|_{L^2(d\sigma)}}{\hat P_\zeta(\hat\zeta)},
\end{equation}
where $d\sigma$ is the sphere surface measure, normalized to satisfy 
$\sigma(\S)=1$ (note that in \cite{BezIII}*{Proposition~2} the sphere is the 
Riemann sphere which has radius $1/2$; we present the result here adapted to the 
unit sphere $\S$). In other words, we have
\begin{equation}\label{def:workingmu2}
\mu_{\rm norm}(P) =\frac12\sqrt{N(N+1)}\max_{1\leq i\leq N} 
\frac{\left(\int_{\S}\prod_{j=1}^N| p-\hat 
z_j|^2d\sigma(p)\right)^{1/2}}{\prod_{j\neq i}|\hat z_i-\hat z_j|}.
\end{equation}
Now we describe the main result in \cite{BezIII}. For a set of points 
$\hat z_1,\ldots,\hat z_N$ in the unit sphere $\S\subseteq\R^3$, we define the 
logarithmic 
energy of these points as
\[
\mathcal{E}(\hat z_1,\ldots,\hat z_N)=\sum_{i\neq j}\log\frac{1}{|\hat z_i-\hat 
z_j|}
\]
(note that in \cite{BezIII} the sum is taken over $i< j$ instead of 
$i\neq j$, which is equivalent to dividing $\mathcal{E}$ by $2$. Here we follow 
the notation in most of the current works in the area). Let
\[
\mathcal{E}_N=\min_{\hat z_1,\ldots,\hat z_N\in\S}\mathcal{E}(\hat 
z_1,\ldots,\hat z_N).
\]
\begin{theorem}[Main result of \cite{BezIII}]\label{th:bez3}
Let $\hat z_1,\ldots,\hat z_N\in\S$ be such that
\[
\mathcal{E}(\hat z_1,\ldots,\hat z_N)\leq \mathcal{E}_N+c\log N.
\]
Let $z_1,\ldots,z_N$ be points in $\mathbb C$ by the inverse stereographic 
projection. Then, the 
polynomial 
$P(z) = \prod_{i=1}^N (z-z_i)$ with zeros $z_1,\ldots, z_N$ satisfies 
$\mu_{\rm norm}(P)\leq 
\sqrt{N^{1+c}(N+1)}$.
\end{theorem}
Theorem~\ref{th:bez3} shows that if one can find $N$ points in the sphere such 
that their logarithmic potential is very close to the minimum then one can  
construct a solution to (the polynomial version of) Problem~\ref{problem:main}. 
Actually, this fact is the reason for the exact form of the problem posed by 
Shub and Smale that is nowadays known as Problem number 7 in Smale's list 
\cite{Smale}:
\begin{problem}[Smale's 7th problem]\label{smale7}
Can one find $\hat z_1,\ldots,\hat z_N\in \S$ such that $\mathcal{E}(\hat 
z_1,\ldots,\hat z_N)\leq \mathcal{E}_N+c\log 
N$ for some universal constant $c$?
\end{problem}
The value of $\mathcal{E}_N$ is not sufficiently well understood. Upper and 
lower bounds were given in \cite{Wagner,RSZ94,Dubickas,Brauchart2008}, and the 
last word is \cite{BS18} where this value is related to the minimum renormalized 
energy introduced in \cite{SS12} proving the existence of a term $C_{\log}\,N$ in 
the assymptotic expansion. The current knowledge is:
\begin{equation}\label{eq:as}
\mathcal{E}_N=\kappa\,N^2-\frac12\,N\log N+C_{\log}\,N+o(N),
\end{equation}
where $C_{\log}$ is a constant and
\begin{equation}\label{conjetura}
\kappa=\int_{\S}\int_{\S}\log |x-y |^{-1}\,d\sigma(x)d\sigma(y)=\frac{1}{2}-\log 2<0
\end{equation}
is the continuous energy. Combining  \cite{Dubickas} with \cite{BS18} it is 
known that
\[
-0.2232823526\ldots\leq C_{\log}\leq 2\log 2 
+\frac12\log\frac23+3\log\frac{\sqrt\pi}{\Gamma(1/3)}=-0.0556053\ldots,
\]
and indeed the upper bound for $C_{\log}$ has been conjectured to be an 
equality using two different approaches \cite{BHS2012b,BS18}.
\subsection{Main result}
Smale's 7th problem seems to be more difficult than the main problem in 
\cite{BezIII}: the 
main result in this paper is a complete solution to the latter. 
More exactly, we have the following result.

\begin{theorem}\label{th:intro}
Given $C_1,C_2>0$ there exists a constant $C>0$ with the following property. Let 
$N\geq1$ and let $M,r_1,\ldots,r_M$ be positive integer numbers such that $M\geq2$ and
\begin{itemize}
\item $N=r_M+2(r_1+\cdots+r_{M-1})$, and
\item $C_1 j\leq r_j\leq C_2 j$ for $1\leq j\leq M$.
\end{itemize}
For $1\leq j\leq M-1$, let $h_j,H_j\in[0,1)$ be defined by
\[
h_j=1-\frac{2}{N}\sum_{k=1}^{j-1}r_k -\frac{r_j}{N},\quad H_j=h_j-\frac{r_j}{N},
\]
and write $r_j=6s_j+rem_j$ where $rem_j\in\{0,\ldots,5\}$ for $2\leq j\leq M$.
Consider the degree $N$ polynomial 
$P_N(z)=P_N^{(1)}(z)P_N^{(2)}(z)P_N^{(3)}(z)P_N^{(4)}(z)$ where
\begin{align*}
P_N^{(1)}(z)=&(z^{4s_M+rem_M}-1)\left(z^{r_1}-\rho(h_1)^{r_1}
\right)\left(z^{r_1}-1/\rho(h_1)^{r_1}\right),\\
P_N^{(2)}(z)=&\left(z^{s_2}-\rho(H_1)^{s_2}
\right)\left(z^{s_2}-1/\rho(H_1)^{s_2}\right),\\
P_N^{(3)}(z)=& 
\prod_{j=2}^{M-1}\left(z^{4s_j+rem_j}-\rho(h_j)^{
4s_j+rem_j}
\right)\left(z^{4s_j+rem_j}-1/\rho(h_j)^{4s_j+rem_j}
\right),\\
P_N^{(4)}(z)=& 
\prod_{j=2}^{M-1}\left(z^{s_j+s_{j+1}}-\rho(H_j)^{
s_j+s_{j+1}}
\right)\left(z^{s_j+s_{j+1}}-1/\rho(H_j)^{s_j+s_{j+1
}}\right),
\end{align*}
where if $s_2=0$ or if $s_{j}+s_{j+1}=0$ the corresponding term is removed from the product and $\rho(x)=\sqrt{(1-x)/(1+x)}$. 
Then, $\mu_{\rm norm}(P_N)\leq C\sqrt{N}$.
\end{theorem}

The reader may note that there is a lot of symmetry in the description of the 
polynomial. Indeed, a very intuitive geometrical description of its zeros will 
be given in Section \ref{sec:construction}.

For a given $N$, there exist in general many choices of $M$ and $r_1,\ldots,r_M$  satisfying the hypotheses of Theorem~\ref{th:intro}. 
For all these choices, the corresponding polynomial satisfies $\mu_{\rm norm}(P_N)\leq C\sqrt{N}$. It is easy to write down different choices with desired properties. For example, one can choose to produce polynomials with rational coefficients or search for the choice that gives, for fixed $N$, the smallest value of $\mu_{\rm norm}$. We now describe a very simple choice that shows 
that 
$M,\ r_1,\ldots,r_M$ can be easily constructed for any $N$. For $t\in(0,\infty)$, by $\lfloor
t\rfloor$ we denote the largest integer that is less than or equal to $t$.
\begin{lemma}\label{lem:erres}
Let $N\geq16$. Then, the following choice of $M,r_1,\ldots,r_M$ satisfies the 
hypotheses of Theorem~\ref{th:intro}.
\begin{itemize}
\item $M=\lfloor\sqrt{N/4}\rfloor\geq2$.
\item $r_j=4j-1$ for $1\leq j\leq M-1$.
\item $r_M=N-2(r_1+\cdots+r_{M-1})=N-4M^2+6M-2$.
\end{itemize}
\end{lemma}

\begin{proof}
The only item to be checked is that, for example, $M\leq r_M\leq 16M$. This is trivially 
implied by the choice of $M$ that guarantees $4M^2\leq N\leq 4M^2+8M+4$.
\end{proof}
The normalized condition number of the polynomials compared to $\sqrt{N}$ corresponding to Lemma \ref{lem:erres} is approximated numerically in Figure \ref{fig:2}.

\begin{figure}[htp]%
	\centering
	\includegraphics[width=\linewidth]{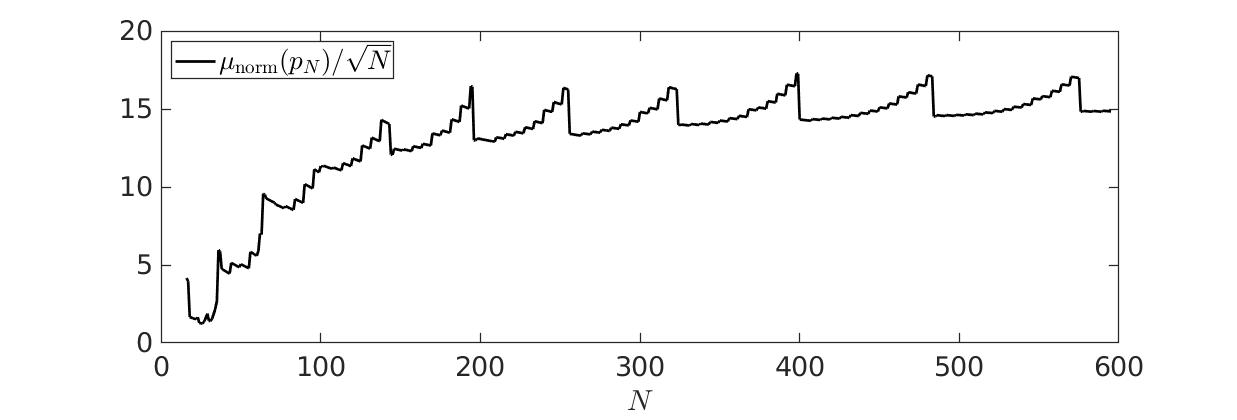}
	\caption{Numerical approximation of $\mu_{\rm norm}(p_N)/\sqrt{N}$ for $p_N$ as in Lemma~\ref{lem:erres} up to degree $595$. The peaks correspond to changes in the value of $M$ as $N$ increases.}\label{fig:2}
\end{figure}

\begin{remark}
Theorem~\ref{th:intro} shows much more than asked in 
Problem~\ref{problem:main} since we get sublinear growth of the condition 
number. The 
presence of the (uncomputed) constant $C$ is not an issue since for all but a 
finite number of values of $N$ we have $C\sqrt{N}\leq N$ and for the first 
values a simple enumeration of the polynomials with rational coefficients will 
produce in finite time a polynomial such that $\mu_{\rm norm}(P)\leq N$. Our 
Theorem~\ref{th:intro} thus fully answers Problem~\ref{problem:main} above.
\end{remark}
\begin{remark}
From Lemma \ref{lem:lb}, the condition number of our sequence of polynomials 
can at most be improved by some constant factor.
\end{remark}

\subsection{Atomization of the logarithmic potential}\label{sec:introgeom}
Theorem~\ref{th:intro} will be proved by atomizing the surface measure in 
$\S$ and approximating the logarithmic potential of the continuous surface 
measure by a potential generated by a measure consisting of equal-weighted 
atoms. This atomization is a well-known technique in
non-harmonic Fourier analysis \cite{LyuSod,LyuSeip}. 

The heuristic argument is 
that if  one places the atoms evenly distributed acording to the surface 
measure, the discrete potential will mimic the  continuous potential which is
constant on the sphere and therefore the numerator and the denominator in 
\eqref{def:workingmu2} will both be very similar. Then, the polynomial whose 
zeros are the inverse stereographic projection of this point set will be well 
conditioned.

Throughout the paper we denote by $C$ a constant that may be different in each 
instance that appears. By $f \lesssim g$ we mean that there 
is a universal constant  $C > 0$ (i.e. independent of $N$) such that $f\le C g$ and we 
write $f \eqsim g $ if there is a universal constant $C>0$ such that 
$C^{-1}f  \le g  \le C f $. 

In Section \ref{sec:construction} of this paper we describe a construction that satisfies the following result.
%
%


\begin{theorem}\label{th:multiplier}
There exists a set $\mathcal P_N$ of $N$ points in $\S$ such that
if $\dist(p,\mathcal P_N)$ denotes the distance from $p\in \S$ to 
$\mathcal P_N$ and 
$\kappa=1/2-\log2.$
Then, for all $p\in\S$ we have $\sqrt{N} \dist(p,\mathcal P_N)\lesssim 1$ and moreover
\begin{equation} 			\label{main_eq}
\sum_{i=1}^N \log  |p-p_i | +\kappa N -\log \left(\sqrt{N} \dist(p,\mathcal  P_N)\right)=O(1).
\end{equation}
Equivalently, 
\begin{equation} 			\label{main_eq2}
\frac{\prod_{i = 1}^N  |p-p_i |^2}{ e^{-2\kappa N} N \dist^2(p, \mathcal P_N)}\eqsim 1, \qquad 
\forall 
p\in\S, \quad \forall N.  
\end{equation}
\end{theorem}

\begin{remark}
	In the case that $p=p_i$ for some $i\in\{1,\ldots,N\}$, \eqref{main_eq2} reads 
\begin{equation*} 		
\frac{\prod_{j\neq i}^N  |p_i-p_j |^2}{e^{-2\kappa N} N}\eqsim 1.
\end{equation*}
\end{remark}

\subsection{Proof of Theorem \ref{th:intro}}
	
	Our main theorem follows immediately from Theorem \ref{th:multiplier} and \eqref{def:workingmu2}. Indeed, we take the polynomial $P_N$ in 
Theorem \ref{th:multiplier} to be  the one whose zeros correspond, under the stereographic projection, to the spherical points $\mathcal P_N$ in 
Theorem \ref{th:multiplier} when the points distributed in each parallel of latitude $t$ are 
rotated to contain the point $(\sqrt{1-t^2},0,t)$. As a result, from \eqref{def:workingmu2}
	\[
	\mu_{\rm norm}(P_N) \lesssim \sqrt{N(N+1)} \frac{\sqrt{N}e^{-\kappa N}\left(\int_{ \S}\dist^2(p,\mathcal P_N)d\sigma(p) \right)^{1/2}}{\sqrt{N}e^{-\kappa N}}\lesssim \sqrt{N}.
	\]

\section{Organization of the paper}

In Section \ref{sec:lowerbound} we prove a sharp lower bound
for the condition number of any polynomial, Lemma \ref{lem:lb}.
In Section \ref{sec:construction} we construct the set of points $\mathcal P_N$ in $\S$ used in Theorem \ref{th:multiplier}
and which give the zeros of the polynomials $P_N$ in Theorem \ref{th:intro}. We study also the separation properties of 
$\mathcal P_N.$  In Section \ref{sec:preliminary} we prove some 
preliminary results comparing the discrete and the continuous potential in a parallel and
the potential in three parallels with the potential in a band. Finally we prove Theorem \ref{th:multiplier} at the end of 
Section \ref{sec:final} as a consequence of the comparison between the discrete potential, the potential in parallels and 
the continuous potential.

%
%


\section{Lower bound for the condition number}\label{sec:lowerbound}

In this section we prove Lemma \ref{lem:lb} 

\begin{proof}

Recall that from (\ref{def:workingmu2})
\[
\mu_{\rm norm}(P)=\frac12\sqrt{N(N+1)}\frac{R}{S},
\]
with 
\[
R=\left(\int_{ \S}\prod_{i=1}^N|p-\hat z_i|\,d\sigma(p)\right)^{1/2},\quad S=\min_{i=1\ldots N}\prod_{j\neq i}|\hat z_i-\hat z_j|.
\]
Here, $P(z)=\prod_{i=1}^N(z-z_i)$ and $\hat z_i$ are the associated points in the unit sphere.
	We bound separately $R$ and $S$. Using Jensen's inequality we have
	\begin{multline*}
	\log R=\frac{1}{2}\log \int_{ \S}\prod_{i=1}^N|p-\hat z_i|^2\,d\sigma(p)\geq 
	\frac{1}{2} \int_{  \S}\log\prod_{i=1}^N|p-\hat z_i|^2\,d\sigma(p)=\\
		\sum_{i=1}^N \int_{  \S}\log|p-\hat z_i|\,d\sigma(p)=-\kappa N,
	\end{multline*}
	and hence $R\geq e^{-\kappa N}$. For bounding $S$, note that from \eqref{eq:as}
	\[
	-\sum_{\stackrel{i,j=1}{i\neq j}}^N\log|\hat z_i-\hat z_j|\geq \kappa N^2-\frac{N}{2}\log N-CN,
	\]
	for some $C>0$. On the other hand,
	\begin{multline*}
	-\sum_{\stackrel{i,j=1}{i\neq j}}^N\log|\hat z_i-\hat z_j|=-\log\left(\prod_{i=1}^N\prod_{j\neq i}|\hat z_i-\hat z_j|\right)\leq -\log(S^N)=-N\log S.
	\end{multline*}
	From
	\[
	-\log(S^N)\geq \kappa N^2-\frac{N}{2}\log N-CN,
	\]
	we get
	\[
	S\lesssim e^{-\kappa N}\sqrt{N},
	\]
proving $R/S\gtrsim 1/\sqrt{N}.$ The lemma follows.
\end{proof}


\section{Construction of the point set $\mathcal P_N$} \label{sec:construction}

In this section, we define the set of points 
$\mathcal P_N=\{p_1,\ldots,p_N\}\subset \S$ appearing in Theorem \ref{th:multiplier}. 
The images of these points through the stereographic projection are 
the zeros of the polynomials in Theorem \ref{th:intro}. The set $\mathcal P_N$ will be a union of equidistributed 
points in symmetric parallels with respect to the $xy$ plane and the construction is similar to the one in \cite{Diamond}.

We denote the parallels in $\S$ by 
$$Q_{h}=\{(x,y,z)\in \S : z=h \},\;\;\;-1\leq h \leq 1.$$

Given $C_1,C_2>0$ and $N\ge 1,$ let $M\geq1$ and 
$r_1,\ldots,r_M\in\Z$ be positive integers such that
\[
N=2(r_1+\ldots+r_{M-1})+r_M,
\]
with 
$$C_1j \le r_j \le C_2 j,$$ for all $1\le j\le M.$

Let
\[
r_{M+1}=r_{M-1},\ldots, r_{2M-1}=r_1. 
\] 
We choose 
parallel heights
$1=H_0>H_1>\cdots>H_{M-1}>0$ and
symmetrically $H_{M+j}=-H_{M-(j+1)}$ for $j=0,\dots ,M-1.$
For $1\le j\le 2M-1$ we define the bands
$$B_j=\{ (x,y,z)\in \S \;\:\; H_j\le z\le H_{j-1} \},$$
where $B_1,B_{2M-1}$ are spherical caps. Then $\S=\bigcup_{j=1}^{2M-1} B_j$ and if we define
\[
H_j=1-\frac{2}{N}\sum_{k=1}^j r_k\quad 0\leq j\leq 2M-1,
\]
we have that
\[
\sigma(B_j)=\frac{H_{j-1}-H_j}{2}=\frac{r_j}{N},\quad 1\leq j\leq 2M-1.
\] 
We consider also parallels with heights
\[
h_j=\frac{H_{j-1}+H_j}{2}=H_{j-1}-\frac{r_j}{N}=H_{j}+\frac{r_j}{N}=1-\frac{2}{N}\sum_{k=1}^{j-1}
r_k 
-\frac{r_j}{N},
\]
for $1\leq j\leq 2M-1,$ and observe that $h_M=0$ and $h_{M+j}=-h_{M-j}$ for $j=1,\dots, M-1.$

Observe that for $1\leq j\leq M$
\begin{equation}\label{eq:zjbound1}
1-\frac{C_2j^2}{N}\leq h_j\leq 1-\frac{C_1j^2}{N},\quad 
1-\frac{C_2j(j+1)}{N}\leq H_j\leq1-\frac{C_1j(j+1)}{N}
\end{equation}
and
\begin{equation}\label{eq:zjbound2}
-1+\frac{C_1j^2}{N}\leq h_{2M-j}\leq -1+\frac{C_2j^2}{N},\quad 
-1+\frac{C_1j(j-1)}{N}\leq H_{2M-j}\leq -1+\frac{C_2j(j-1)}{N}.
\end{equation}
Note that we have
\[
C_1M^2= C_1M+2\sum_{j=1}^{M-1}C_1j\leq  N\leq C_2M+2\sum_{j=1}^{M-1}C_2j\leq 
C_2M^2.
\]

 We say that a set of points are equidistributed in a parallel if they are, up to 
homotety, rotation and traslation, a set of roots of unity in the circle defined by the parallel. 
Given the points $r_j$ above we define $\tilde r_1=\tilde r_{2M-1}=0$
and $r_j=\tilde r_j+rem_j$ for $2\leq j\leq 2M-2$ where $\tilde r_j$ is a multiple of $6$ and $0\leq rem_j\leq 5.$
Note that in Theorem \ref{th:multiplier} we denote $\tilde r_j=6s_j$.
Then to define the set $\mathcal P_N$ 
\begin{itemize}
\item we take $r_1$ points equidistributed in $Q_{h_1}$ and similarly $r_{2M-1}=r_1$ points equidistributed in $Q_{h_{2M-1}}=Q_{-h_1}$.
\item For $2\leq j\leq 2M-1$, we take $\frac{4\tilde{r}_j}{6}+rem_j$ points equidistributed at $Q_{h_j},$ 
$\frac{\tilde{r}_{j-1}+\tilde{r}_{j}}{6}$ points equidistributed in the upper boundary parallel $Q_{H_{j-1}}$ and for $1\leq j\leq 2M-2$ 
we take $\frac{\tilde{r}_j+\tilde{r}_{j+1}}{6}$ 
points equidistributed in the lower boundary parallel $Q_{H_{j}}.$
\end{itemize}
Observe that in this way there are $r_j$ points of $\mathcal P_N$ in
the band $B_j$ for $j=1,\dots, 2M-1.$



\subsection{Geometric properties of the set $\mathcal P_N$}

From the results in this section it follows that the points in $\mathcal P_N$ are uniformly separated i.e. for each 
$p,q\in \mathcal P_N$ distinct
$$\dist (p,q)\gtrsim 1/\sqrt{N},$$
and they are relatively dense i.e.
for all $p\in \S$ we have that
$$\dist (p,\mathcal P_N)\lesssim 1/\sqrt{N}.$$

%

\begin{lemma}\label{lem:dist}
	For $h,c\in(-1,1)$ with $|h|\leq|c|$ and $|h-c|\leq1/4$ we have
	\[
\dist(Q_c,Q_h)\eqsim
	\frac{|h-c|}{\sqrt{1-h^2}}\lesssim\frac{|h-c|}{\sqrt{1-c^2}}.
	\]
\end{lemma}
\begin{proof}
	Note that $\dist(Q_c,Q_h)\leq 1$ and we can write also
	\[
	\dist(Q_c,Q_h)=2\sin\frac{\varphi}{2},
	\]
	where $\varphi$ is the angular distance from $Q_c$ to 
	$Q_h$. Moreover,
	\[
	\varphi=2\arcsin\frac{\dist(Q_c,Q_h)}{2}\leq 2\arcsin\frac12=\frac{\pi}{3}.
	\]
	We first prove the lower bound. Note that for $\varphi\in[0,\pi/3]$
	\[
	\dist(Q_c,Q_h)=2\sin\frac{\varphi}{2}\geq\frac{\varphi}{2}\gtrsim|\arcsin(h)-\arcsin(c)|=\frac{|h-c|}{\sqrt{1-\zeta^2}}
	\]
	for some $\zeta$ in the interval containing $c$ and $h$. Now, if $c$ and $h$ have both the same sign then $\sqrt{1-\zeta^2}\leq \sqrt{1-h^2}$ and we are done. Moreover, if $|h|\leq1/2$ then $|c|\leq3/4$ and $\sqrt{1-\zeta^2}\eqsim1\eqsim\sqrt{1-h^2}$. These are all the cases to cover since $|h-c|\leq1/4$ excludes other situations. We have proved that $\dist(Q_c,Q_h)\gtrsim|h-c|/\sqrt{1-h^2}$. 
	
	For the upper bound, again using the same argument we can assume that $1/2\leq h\leq c\leq 1$. Then,
	\begin{multline*}
	2\sin\frac{\varphi}{2}\lesssim\sin\varphi=|\sin(\arcsin(h)-\arcsin(c))|=|h\sqrt{1-c^2}-c\sqrt{1-h^2}|=\\\frac{c^2-h^2}{h\sqrt{1-c^2}+c\sqrt{1-h^2}}\lesssim\frac{c-h}{\sqrt{1-h^2}}.
	\end{multline*}

\end{proof}


\begin{lemma}\label{lem:facil2}
The distance between two points of $\mathcal P_N$ in the same parallel is or order $1/\sqrt{N}$, i.e.
\[
\frac{\sqrt{1-h_j^2}}{r_j}\eqsim \frac{1}{\sqrt{N}},\quad \frac{\sqrt{1-H_j^2}}{r_j}\eqsim \frac{1}{\sqrt{N}}
\]
where the first claim is valid for $1\leq j\leq 2M-1$, and the second one is valid for $1\leq j\leq 2M-2.$ 
In particular, this implies
\[
\frac{1-h_j^2}{1-H_j^2}\eqsim 1,\quad 1\leq j\leq 2M-2
\]
and similarly
\[
\frac{1-h_j^2}{1-H_{j-1}^2}\eqsim 1,\quad 2\leq j\leq 2M-1.
\]
\end{lemma}


\begin{proof}
By symmetry we can assume that $j\leq M$. Then, $h_j\geq0$ and hence $\sqrt{1-h_j^2}\eqsim\sqrt{1-h_j}$, which from \eqref{eq:zjbound1} yields
\begin{align*}
\frac{\sqrt{1-h_j^2}}{r_j}\eqsim \frac{\sqrt{1-h_j}}{j} \eqsim\frac{1}{\sqrt{N}}.
\end{align*}
The inequality for $H_j$ is proved in a similar way.
\end{proof}


\begin{lemma}\label{lem:facil3}
	The distance between consecutive parallels is of order $1/\sqrt{N}$, i.e.
\[
\dist(Q_{H_{j-1}},Q_{h_j})\eqsim \frac{1}{\sqrt{N}},\quad 
\dist(Q_{H_{j}},Q_{h_j})\eqsim \frac{1}{\sqrt{N}}.
\]
\end{lemma}
\begin{proof}
	By symmetry, we can assume that $h_j\geq0$, that implies $|h_j|\leq|H_{j-1}|$.
	From Lemmas \ref{lem:dist}  and \ref{lem:facil2} we have
	\[
	\dist(Q_{H_{j-1}},Q_{h_j})\eqsim \frac{r_j/N}{\sqrt{1-h_j^2}}\eqsim\frac1{\sqrt{N}}.
	\]
	The other inequality is proved in a similar way.
\end{proof}



\section{Comparison of discrete potentials, parallels and bands}\label{sec:preliminary}

For $-1\le h\le 1$ and $p\in \S$ we denote 
$$f_p(h)=\int_{0}^{2\pi}\log |p-\gamma_h(\theta) |\,\frac{d\theta}{2\pi},$$
where
$\gamma_h(\theta)=(\sqrt{1-h^2}\cos\theta,\sqrt{1-h^2}\sin\theta,h).$ 
In words, $f_p(h)$ is the mean value of $\log|p-q|$ when $q$ lies in the parallel $Q_h$.
For $-1\le c,z\le 1$ we denote
\[
R(c,z)=6\sqrt{1-c^2}(1-z^2)+8(1-z^2)^{3/2}.
\]


\begin{lemma}\label{lem:deri2}
Let $\gamma_h(\theta)=(\sqrt{1-h^2}\cos\theta,\sqrt{1-h^2}\sin\theta,h)$, 
$\theta\in[0,2\pi]$ be a parame\-tri\-za\-tion of $Q_h$, and let 
$p=(a,b,c)\in\S\setminus Q_h$. Then,
\[
\left|\frac{d^3}{d\theta^3}\log|p-\gamma_h(\theta)|\right|\leq\frac{\sqrt{1-h^2}
}{|p-\gamma_h (\theta)|}+\frac{R(c,h)}{|p-\gamma_h (\theta)|^3}.
\]
\end{lemma}
\begin{proof}
We can assume that $p=(\sqrt{1-c^2},0,c)$ and denote $\gamma_h=\gamma$. Let 
$F(\theta)=\log|p-\gamma(\theta)|$ and note that, as $\langle 
\gamma'(\theta),\gamma(\theta)\rangle=0$
\[
F'(\theta)=-\frac{\langle 
p-\gamma(\theta),\gamma'(\theta)\rangle}{|p-\gamma(\theta)|^2}=-\frac{\langle 
p,\gamma'(\theta)\rangle}{|p-\gamma(\theta)|^2},
\]
\begin{equation}\label{eq:2der}
F''(\theta)=-\frac{\langle 
p,\gamma''(\theta)\rangle}{|p-\gamma(\theta)|^2}-\frac{2\langle 
p,\gamma'(\theta)\rangle^2}{|p-\gamma(\theta)|^4},
\end{equation}
and
\[
F'''(\theta)=-\frac{\langle 
p,\gamma'''(\theta)\rangle}{|p-\gamma(\theta)|^2}-\frac{6\langle 
p,\gamma''(\theta)\rangle\langle 
p,\gamma'(\theta)\rangle}{|p-\gamma(\theta)|^4}-\frac{8\langle 
p,\gamma'(\theta)\rangle^3}{|p-\gamma(\theta)|^6}.
\]
Now,
 \begin{align}            \label{eq:normas}
  |\langle p,\gamma'(\theta)\rangle| & =
|\langle p-\langle 
p,\gamma(\theta)\rangle\gamma(\theta),\gamma'(\theta)\rangle| \nonumber
\\
&
\leq 
\sqrt{1-\langle p,\gamma(\theta)\rangle^2}|\gamma'(\theta)|\leq 
|p-\gamma(\theta)|\sqrt{1-h^2},
 \end{align}
and since $\gamma'''=-\gamma'$ the same bound holds changing $\gamma'$ to 
$\gamma'''$. Finally, note that
\begin{equation}\label{eq:2der2}
|\langle p,\gamma''(\theta)\rangle|\leq\sqrt{1-c^2}\sqrt{1-h^2},
\end{equation}
and the lemma follows.
\end{proof}




\begin{lemma}[Comparison of the finite sum with the integral along the parallel]\label{lem:pointstoparallels}
Assume that $\dist(p,Q_h)\gtrsim \sqrt{1-h^2}/A.$ Let $q_i\in Q_h$
for $i=1,\dots , A$ be points at angular distance $2\pi/A.$ 
Then
\begin{align}						\label{bound_comparison1}
\big|  \sum_{i=1}^A   &   \log |p-q_i |  -Af_p(h)\big|  
\\ 
&
\lesssim \frac{1}{A^2}\left(\sqrt{1-h^2}
\int_{0}^{2\pi}\frac{1}{ |p-\gamma_h(\theta) |}\, d\theta +R(c,h)\int_{0}^{2\pi}\frac{1}{ |p-\gamma_h(\theta) |^3}\,d\theta \right). \nonumber
\end{align}
Moreover, if $B\supseteq Q_h$ is a band of height $\epsilon\lesssim (1-h^2)/A$ and such that $d(p,B)\gtrsim \frac{\sqrt{1-h^2}}{A}$ then
\begin{equation*}
 \int_{0}^{2\pi}\frac{1}{ |p-\gamma_h(\theta) |}\, d\theta\eqsim \frac{1}{\sigma(B)}\int_{B}\frac{1}{ |p-q |}\, d\sigma(q)
\end{equation*}
and
\begin{equation*}
 \int_{0}^{2\pi}\frac{1}{ |p-\gamma_h(\theta) |^3}\, d\theta\eqsim \frac{1}{\sigma(B)}\int_{B}\frac{1}{ |p-q |^3}\, d\sigma(q),
\end{equation*}
where the constants are independent of $h$ and $A.$
\end{lemma}


\begin{proof} Without loss of generality, we can assume that $h\geq0$ and $q_i=\gamma_h(\theta_i)$ with 
$\theta_i=(2i-1)\pi/A.$ Define the periodic function $\phi(\theta)=\log |p-\gamma_h(\theta)|.$ Since $\phi'(\theta)$ is also periodic
(\ref{bound_comparison1}) equals
\begin{align*}			
\left|\sum_{i=1}^A  \right. & \left.  \phi(\theta_i)  -A \int_{0}^{2\pi}\phi(\theta)\,
\frac{d\theta}{2\pi}+\frac{\pi}{12A}\int_{0}^{2\pi}\phi''(\theta)\,d\theta\right|
  \\
&
\le
\sum_{i=1}^A\left|\phi(\theta_i)-\frac{A}{2\pi}\int_{I_i}\phi(\theta)\,d\theta+
\frac{\pi}{12A}\int_{I_i}\phi''(\theta)\,d\theta\right| 
\\
&
\lesssim
\frac{1}{A^3}\sum_{i=1}^A\sup_{\theta\in I_i}
|\phi'''(\theta)|\le 
\frac{1}{A^3}\sum_{i=1}^A \sup_{\theta\in I_i} \left(  \frac{\sqrt{1-h^2}}{
 |p-\gamma_h(\theta) |}+\frac{R(c,h)}{|p-\gamma_h(\theta) |^3} \right)
\end{align*}
by Lemma~\ref{lem:midpoint} and Lemma \ref{lem:deri2} where $I_i=[\theta_i-\pi/A,\theta_i+\pi/A].$
Let $\theta,\theta' \in I_i$ be two points were 
$|p-\gamma_h(\cdot)|$ attains respectively its minimum and its maximum value. 
Then,
\begin{align*}
|p-\gamma_h(\theta')|\leq & |p-\gamma_h(\theta) |+ |\gamma_h(\theta)-\gamma_h( 
\theta') |\leq |p-\gamma_h(\theta)|+\frac{2\pi\sqrt{1-h^2}}{A}
\\
\lesssim
& 
|p-\gamma_h(\theta) |+ \dist(p,Q_h) \lesssim |p-\gamma_h(\theta) |,
\end{align*}
and 
$$\sup_{\theta\in I_i}\left(  \frac{\sqrt{1-h^2}}{
 |p-\gamma_h(\theta) |}+\frac{R(c,h)}{|p-\gamma_h(\theta) |^3} \right) \lesssim
A \left(\int_{I_i} \frac{\sqrt{1-h^2}}{|p-\gamma_h(\theta) |}d\theta +\int_{I_i} \frac{R(c,h)}{|p-\gamma_h(\theta) |^3}d\theta\right). $$

Now we prove the second part of the lemma. Assume that the band $B$ is the set contained between $Q_{h_0}$ and $Q_{h_0+2\epsilon}$. 
For $q\in B$ let $q'\in Q_h$ be the closest point to $q$ in $Q_h$. Then, from Lemma \ref{lem:dist} we have that $|q-q'|\lesssim\epsilon/\sqrt{1-h^2}$ and hence
$$|p-q|\le |p-q'|+|q'-q|\lesssim |p-q'|+\frac{\sqrt{1-h^2}}{A}\lesssim |p-q'|,$$
and similarly
$$
|p-q'|\le |p-q|+|q'-q|\lesssim |p-q|+\frac{\sqrt{1-h^2}}{A}\lesssim  |p-q|.
$$
In other words, we have $|p-q|\eqsim |p-q'|$ and therefore
$$\int_B \frac{1}{|p-q|}d\sigma(q)=\frac1{4\pi}\int_{h_0}^{h_0+2\epsilon}  \int_{0}^{2\pi} \frac{1}{|p-\gamma_t(\theta)|}d\theta dt
\eqsim  \epsilon \int_{0}^{2\pi} \frac{1}{|p-\gamma_h(\theta)|}d\theta,$$
and we conclude the result after an identical reasoning for the integral of $|p-q|^{-3}$.

\end{proof}


\begin{lemma}[Computation of the integral along one 
parallel]\label{lem:lineintegral}
Let $p=(a,b,c)\in \S$. Then,
\[
f_p(h)
=\frac12\log(1-hc+|h-c|)=\left\{
\begin{array}{ll}
\frac{1}{2}(\log(1+h)+\log(1-c))  & \mbox{if } h\ge c, \\ \\
\frac{1}{2}(\log(1-h)+\log(1+c)) & \mbox{if } h< c.
\end{array}
\right.
\]
\end{lemma}
\begin{proof}
See \cite[4.224.9]{LosRusos}.
\end{proof}

\begin{lemma}\label{lem:lineintegralcuad}
	Let $p=(a,b,c)\in \S$.
	The following equality holds
	\[
	\int_{0}^{2\pi}\frac{1}{|p-\gamma_h(\theta) |^2}\,\frac{d\theta}{2\pi}
	=\frac{1}{2|h-c|}.
	\]

\end{lemma}
\begin{proof}
	From \cite[3.661.4]{LosRusos} we have
	\begin{align*}
	\int_{0}^{2\pi}\frac{1}{|p-\gamma_h(\theta) |^2}\,\frac{d\theta}{2\pi}=&\int_{0}^{\pi}\frac{1}{2-2ch-2\sqrt{1-h^2}\sqrt{1-c^2}\,\cos\theta}\,\frac{d\theta}{\pi}\\=&\frac{1}{\sqrt{\left(2-2ch\right)^2-\left(2\sqrt{1-h^2}\sqrt{1-c^2}\right)^2}},
	\end{align*}
and the lemma follows after expanding the denominator.
\end{proof}


\begin{lemma}[Comparison of integrals on parallels and bands]\label{lem:comparsionparalelband}
Let $B$ be the band containing $Q_h$ given by $B=\{q\in\S:\langle 
q,e_3\rangle\in[h-\epsilon,h+\epsilon]\}.$ Assume that $h-\epsilon,h+\epsilon\in(-1,1)$ and let $p\in\S\setminus B$. Then
\begin{equation} 		\label{parallelband1}
\begin{split}
\left|f_p(h)-\frac{1}{\epsilon}\int_{B}\log |p-w |\,d\sigma(w)\right|\lesssim
\frac{\epsilon^2}{(1-\max(|h-\epsilon|,|h+\epsilon|)^2)^2},
\end{split}
\end{equation}
and
\begin{equation}    \label{parallelband2}
\begin{split}
\left|\frac{f_p(h-\epsilon)+4 f_p (h)+f_p (h+\epsilon)}{6} -\frac{1}{\epsilon}\int_{B}\log |p-w |\,d\sigma(w)\right|
\lesssim
\frac{\epsilon^4}{(1-\max(|h-\epsilon|,|h+\epsilon|)^2)^4}.
\end{split}
\end{equation}
\end{lemma}


\begin{proof}
Using that
\[
\frac{1}{\epsilon}\int_{B}\log |p-w |\,d\sigma(w)=\frac{1}{2\epsilon}\int_{h-\epsilon}^{h+\epsilon}f_p(t) dt,
\]
and Lemma~\ref{lem:lineintegral},
the results follows from the error estimation for the midpoint integral rule
and for the Simpson rule, see Lemma \ref{lem:midpointbasico}. Note that we are also using
\[
1-\max(|h-\epsilon|,|h+\epsilon|)\eqsim 1-\max(|h-\epsilon|,|h+\epsilon|)^2.
\]
\end{proof}


\begin{lemma}[Comparison of the integrals on the parallel and the band: the 
case 
that the band contains the point $p=(a,b,c)$]\label{lem:comparsionparalelband2}
Let $B$ be the band containing $Q_h$ given by $B=\{q\in\S:\langle 
q,e_3\rangle\in[h-\epsilon,h+\epsilon]\}$. 
Here, we are assuming that $h-\epsilon,h+\epsilon\in(-1,1)$. Then, if 
$h-\epsilon\leq c\leq h+\epsilon$,
\begin{equation} 		\label{parallelband1point}
\begin{split}
\left|
f_p(h)-\frac{1}{\epsilon}
\int_{B}\log |p-w |\,d\sigma(w)\right|\lesssim \frac{\epsilon}{1-c^2}.
\end{split}
\end{equation}
and
\begin{equation}    \label{parallelband2point}
\begin{split}
\left|\frac{1}{6}\left(f_p(h-\epsilon)+4 f_p (h)+f_p (h+\epsilon)\right)-\frac{1}{\epsilon}\int_{B}\log |p-w |\,d\sigma(w)\right|\lesssim \frac{\epsilon}{1-c^2}.
\end{split}
\end{equation}
\end{lemma}


\begin{proof}
As in the proof of Lemma~\ref{lem:comparsionparalelband}, note that
\[
\frac{1}{\epsilon}\int_{B}\log |p-w | d\sigma(w)=
\frac{1}{2\epsilon}\int_{h-\epsilon}^{h+\epsilon}f_p(t)dt
=f_p(h)+\frac{1}{2\epsilon}\int_{h-\epsilon}^{h+\epsilon}(f_p(t)-f_p(h))\,dt.
\]
Then, the quantity in (\ref{parallelband1point}) can be bounded by $2 \epsilon$ times the Lipschitz constant $L_{f_p}$ of $f_p.$ 
By Lemma~\ref{lem:lineintegral}  
$$L_{f_p}\lesssim \max \left\{  \sup_{t\in [h-\epsilon,c]}\frac{1}{1-t} ,\sup_{t\in [c,h+\epsilon]}\frac{1}{1+t}  \right\}\lesssim
\frac{1}{1-c}+ \frac{1}{1+c}$$
and (\ref{parallelband1point}) follows.

In (\ref{parallelband2point}) we decompose the Simpson's rule in the midpoint and the trapezoidal rules. For the midpoint
we do as before. For the trapezoidal rule let $\ell(t)$ the line through $(h-\epsilon,f(h-\epsilon))$ and  $(h+\epsilon,f(h+\epsilon)).$
To estimate
$$\frac{1}{2\epsilon}\int_{h-\epsilon}^{h+\epsilon} (\ell(t)-f_p(t)) dt,$$
we use that for  $h-\epsilon\le t \le h+\epsilon$
$$|\ell(t)-f_p(t)|\lesssim (L_{\ell}+L_f) \epsilon$$
and clearly $L_\ell \le L_f.$
\end{proof}



\section{The proof of Theorem \ref{th:multiplier}} \label{sec:final}

The strategy to prove Theorem \ref{th:multiplier} will follow two steps. First we 
approximate the potential generated by the surface measure in $\S$ by a potential 
generated by a multiple of the length-measure supported in several chosen 
parallels $Q_{h_j}$ and $Q_{H_j}.$ Then, we compare the potential in parallels 
with the discrete potential given by the points in $\mathcal P_N.$ We follow the notation from
Section \ref{sec:construction}.

\subsection{From bands to parallels}

We show that, given $p\in\S$, the mean value of $N\log|p-q|$ for $q\in\S$ is 
comparable to the weighted sum of the mean values in different parallels $Q_{h_j},Q_{H_j}$ where the weights 
are given by the number of points that we have placed in each parallel.

\begin{proposition}									\label{bandparallel}
Let $p=(a,b,c)\in\S$ and let $\mathcal P_N$ be a collection of $N$ points as defined in 
Section~\ref{sec:construction}. 
Let
\begin{align*}
S_N  =r_1(f_p(h_1) & +f_p(h_{2M-1}))  +\sum_{j=2}^{2M-2}\left( \frac{4 \tilde{r}_j}{6}+rem_j \right)f_p(h_j)
\\
&
+\sum_{j=1}^{2M-2}\left( \frac{\tilde{r}_j
+\tilde{r}_{j+1}}{6} \right)f_p(H_j)
\end{align*}
 Then,
\[
\left|S_N-N\int_{\S }\log |p-q |\,d\sigma(q)\right|\lesssim 1.
\]

\end{proposition}


\begin{proof}
  Assume that $p\in B_\ell$ and $\ell \neq 1,2M-1.$ Then we can write the difference above as
  \begin{align}
   \sum_{\stackrel{j=2}{j\neq \ell}}^{2M-2} & \left[ \frac{\tilde{r}_j}{6}(f_p(H_{j-1})+4 f_p(h_j) +f_p(H_j))-\frac{\tilde{r}_j}{\sigma(B_j)}\int_{B_j}\log |p-q|d\sigma(q)\right] \label{x1}
   \\
   &
   \sum_{\stackrel{j=2}{j\neq \ell}}^{2M-2} \left[ rem_j f_p(h_j)-\frac{rem_j}{\sigma(B_j)}\int_{B_j}\log |p-q|d\sigma(q)\right]         \label{x1b}
   \\
   &
   +rem_\ell f_p(h_\ell)-\frac{rem_\ell}{\sigma(B_\ell)}\int_{B_\ell}\log |p-q|d\sigma(q)   \label{x2}
   \\
   &
   +    \frac{\tilde{r}_\ell}{6}(f_p(H_{\ell-1})+4 f_p(h_\ell) +f_p(H_\ell))-\frac{\tilde{r}_\ell}{\sigma(B_\ell)}\int_{B_\ell}\log |p-q|d\sigma(q)   \label{x3}
   \\
   &
   +r_1 (f_p(h_1)+f_p(-h_1))-\frac{r_1}{\sigma(B_1)}\int_{B_1\cup B_{2M-1}}\log |p-q|d\sigma(q).   \label{x4}
  \end{align}
  For the first sum (\ref{x1}) we use (\ref{parallelband2}) with $\epsilon=r_j/N$ and using from Lemma \ref{lem:facil2} that $j^2\eqsim r_j^2\eqsim N(1-H_j^2)\eqsim N(1-h_j^2)\eqsim N(1-H_{j-1}^2)$ we get
$$ (\ref{x1})\lesssim  \sum_{j=2}^{2M-2} \frac{r_j^5}{N^4 (1-h_j^2)^4}\lesssim \sum_{j=1}^{\infty}\frac{1}{ j^3 }\eqsim 1.$$

  For (\ref{x1b}) we apply (\ref{parallelband1}) with $\epsilon=r_j/N$ and Lemma \ref{lem:facil2}  again. Using also that $rem_j<6$ and $j^2\eqsim r_j^2\eqsim N(1-h_j^2)$ we get 
$$ (\ref{x1b})\lesssim \sum_{j=2}^{2M-2}\frac{r_j^2}{N^2(1-h_j^2)^2}\lesssim \sum_{j=1}^{\infty} \frac{1}{j^2}\eqsim 1.$$

For (\ref{x2}) and (\ref{x3}) we use (\ref{parallelband1point}) and (\ref{parallelband2point}). This, together with Lemma \ref{lem:facil2} yields
$$(\ref{x2})+(\ref{x3})\lesssim \frac{\tilde{r}_\ell}{N(1-c^2)}+\frac{\tilde{r}_\ell^2}{N(1-c^2)}\lesssim \frac{1}{\ell}+1\eqsim 1.$$

Finally, (\ref{x4})$\eqsim1$ as follows from Lemma \ref{lem:inthenorth}.
\end{proof}

%


\begin{lemma}  \label{lem:inthenorth}
 For any $p\in \S$ we have 
 $$\left| f_p(h_1) -\frac{1}{\sigma(B_1)}\int_{B_1 }\log |p-q|d\sigma(q) \right|\lesssim 1.$$
\end{lemma}
\begin{proof}
This follows from a direct computation. If $p\not\in B_1$ then the quantity in the lemma is
\[
\left|\frac{1}{2}\log(1+h_1)-\frac{N}{2}\log2+\frac{N}{2}\log\left(2-\frac{2r_1}{N}\right)+\frac12\log\left(2-\frac{2r_1}{N}\right)\right|\lesssim1,
\]
since $\log\left(2-\frac{2r_1}{N}\right)-\log2=\log\left(1-\frac{2r_1}{N}\right)\eqsim1/N$.
If $p\in B_1$ it is a little longer computation. One must write
\[
\int_{B_1 }\log |p-q|d\sigma(q)=\int_{1-2r_1/N}^1f_p(t)\,dt,
\]
and consider two subintervals depending on $t<\langle p,e_3\rangle$ or $t>\langle p,e_3\rangle$. Then, from Lemma \ref{lem:lineintegral} this quantity can be computed exactly and the lemma follows after some elementary manipulations.
%
%
%
\end{proof}


\subsection{From points to parallels}

In this section we prove Theorem \ref{th:multiplier}. 
 Recall that the sum for all parallels $S_N$ defined in Proposition \ref{bandparallel}.
Then,
\begin{align*}
\sum_{p_i\in \mathcal P_N} & \log |p-p_i|-S_N \nonumber
\\
&
=\sum_{p_i\in Q_{h_1}\cup Q_{h_{2M-1}}}\log |p-p_i|-r_1(f_p(h_1)  +f_p(h_{2M-1}))    
\\
&
+\sum_{j=2}^{2M-2}\left[\sum_{p_i \in Q_{h_j}} \log|p-p_i|-\left( \frac{4 \tilde{r}_j}{6}+rem_j \right)f_p(h_j) \right] 
\\
&
+\sum_{j=1}^{2M-2}\left[\sum_{p_i \in Q_{H_j}} \log|p-p_i|-  \left( \frac{\tilde{r}_j+\tilde{r}_{j+1}}{6} \right)f_p(H_j)\right]. 
\end{align*}

We will bound in a different way the terms corresponding to three situations: that the parallel ($Q_{h_j}$ or $Q_{H_j}$) is very close to $p$, moderately close to $p$ and far away from $p$.


\subsubsection{The closest parallel}
We will bound the term corresponding to the parallel containing the closest point to $p$ using the following lemma. If there is more than one parallel with this property, we can apply the lemma to any of them.
\begin{lemma}\label{lem:veryclose}
	Let $p\in\S$ and let $p_{i_0}\in \mathcal P_N$ be the closest point to $p$. Assume that $p_{i_0}\in Q_{h_\ell}$. Then
	\[
	\left|\sum_{\stackrel{p_i \in Q_{h_\ell}}{i\neq i_0}} \log|p-p_i|-\left( \frac{4 \tilde{r}_\ell}{6}+rem_\ell \right)f_p(h_\ell)-\log\sqrt{N}\right|\lesssim1.
	\]
Similarly, if $p_{i_0}\in Q_{H_\ell}$, then
	\[
\left|\sum_{\stackrel{p_i \in Q_{H_\ell}}{i\neq i_0}} \log|p-p_i|-\left( \frac{\tilde{r}_\ell+\tilde r_{\ell+1}}{6} \right)f_p(H_\ell)-\log\sqrt{N}\right|\lesssim1.
\]
\end{lemma}
\begin{proof}
	Since the proof of both inequalities is equal, we just prove the first one and we use the notation $Q_\ell=Q_{h_\ell}$, $\gamma_\ell=\gamma_{h_\ell}$ and $c_\ell=4\tilde r_\ell/6+rem_\ell\eqsim \ell.$
	We rename $\mathcal P_N\cap Q_{\ell}=\{ q_1,\dots , q_{c_{\ell}} \}$ and we call 
	$q_1$ the closest point to $p,$ with the former notation, 
	$p_{i_0}=q_1.$   We split the parallel $Q_{\ell}$ in arcs $\gamma_{\ell}(I_j)$ centered 
	on each $q_j$ with angle $\frac{2\pi }{c_{\ell}}$. With this notation, the sum in the lemma --without the $\log\sqrt{N}$ term-- is
	\begin{align*}
		\sum_{\stackrel{p_i \in Q_{\ell}}{i\neq i_0}}  \log | 
		p-p_i|-c_\ell f_p(h_\ell)
		=&
		\sum_{j=2}^{c_{\ell}}\log | p-q_j|-
		\frac{c_{\ell}}{2\pi} \sum_{j=2}^{c_{\ell}} \int_{I_j}\log | p-\gamma_{\ell}(\theta)|d\theta
		\\
		&
		-\frac{c_{\ell}}{2\pi}  \int_{I_1}\log | p-\gamma_{\ell}(\theta)|d\theta.
	\end{align*}
	
	First we estimate this last integral.
	By a rotation we assume that $\gamma_{\ell}(I_1)$ is centered at the point 
	$\tilde{q}_1=(\sqrt{1-h_{\ell}^2},0,h_{\ell})$ and we denote the rotated arc 
	by $I.$ By this rotation the point $p$ goes to some other point 
	$\tilde{p}.$
	Observe that to estimate the integral
	$$-\frac{c_{\ell}}{2\pi} \int_{I}\log | \tilde{p}-\gamma_{\ell}(\theta)|d\theta,$$
	from above, we can replace $\tilde{p}$ by the point $\tilde{q}_1.$
	Indeed,
	$$\Delta_u (\log | u-q |)=2\pi \delta_q-2\pi d\sigma,$$ where 
	$\Delta_u$ is the Laplace-Beltrami operator with respect to the 
	variable $u$ and $\delta_u$ is Dirac's delta. Therefore, out of $q\in \gamma_\ell (I),$ the function $-\log | r-q 
	|$ is subharmonic and satisfies 
	the maximum principle 
	$$\sup_{u\in \S \setminus I}	\int_{I} \log \frac{1}{| u-\gamma_{\ell}(\theta)|}d\theta \le 
	\sup_{ u\in I}	\int_{I} \log \frac{1}{| u-\gamma_{\ell}(\theta) |}d\theta .$$
	Clearly, this last integral is smaller that 
	$$\int_{I} \log \frac{1}{| \tilde{q}_1-\gamma_{\ell}(\theta)|}d\theta.$$
	Using this observation we get for some constant $C>0$ (whose value may vary en each appearance):
	\begin{align*}
		-\frac{c_{\ell}}{2\pi}  \int_{I} \log | \tilde{p}-\gamma_{\ell} (\theta)| 
		d\theta&\leq  -\frac{c_{\ell}}{2\pi} \int_{I} \log | \tilde{q}_1-\gamma_{\ell} (\theta)| 
		d\theta  =-\frac{c_{\ell}}{4\pi}\int_{-\frac{\pi}{c_{\ell}}}^{\frac{\pi}{c_{\ell}}}
		\log (1-\cos \theta) d\theta
		\\
		&
		-\frac{1}{2}\log 
		(1-h_{\ell}^2)+C\le-\frac{c_{\ell}}{\pi}\int_{0}^{\frac{\pi}{c_{\ell}}}\log \theta
		d\theta-\frac{1}{2}\log (1-h_{\ell}^2)+C
		\\
		&
		\le\log \frac{c_{\ell}}{\sqrt{1-h_{\ell}^2}}+C\le\log \sqrt{N}+C,
	\end{align*}
	where we use that $-\log(1-\cos \theta)\leq \log(4/\theta^2)$ in the range $\theta\in[-\pi/2,\pi/2]$ and 
	Lemma~\ref{lem:facil2}. We also have a similar lower bound coming from the fact that $|p-\gamma_{\ell}(\theta)|\lesssim1/\sqrt{N}$ for $\theta\in I_1$:
	\[
	-\frac{c_{\ell}}{2\pi}  \int_{I} \log | \tilde{p}-\gamma_{\ell} (\theta)| 
	d\theta\geq \frac{c_{\ell}}{2\pi}  \int_{I} \log \sqrt{N}d\theta-C=\log\sqrt{N}-C.
	\]
	In other words, we have proved that
	\[
	\left|-\frac{c_{\ell}}{2\pi}  \int_{I} \log | \tilde{p}-\gamma_{\ell} (\theta)| 
	d\theta-\log\sqrt{N}\right|\lesssim1.
	\]
	A similar argument shows that
		\[
		\left|-\frac{c_{\ell}}{2\pi}  \int_{I_j} \log | \tilde{p}-\gamma_{\ell} (\theta)| 
	d\theta-\log|p-q_j|\right|\lesssim1
	\]
	for any $j\neq1$. This allows us to remove any constant number of terms of the sum in the lemma for proving the bound. We thus have to bound
	\begin{equation}\label{eq:ultima}
		\left|\sum_{j\in J}\log | p-q_j|-
		\frac{c_{\ell}}{2\pi} \sum_{j\in J} \int_{I_j}\log | p-\gamma_{\ell}(\theta)|d\theta\right|,
		\end{equation}
		where $J$ is the set of indices $j\in\{1,\ldots,c_\ell\}$ such that $\dist(p,\gamma_{\ell}(I_j))\geq 1/\sqrt{N}$. Now, for such $j$ we can apply the classical estimate for the midpoint rule in Lemma \ref{lem:midpointbasico} getting
		\[
		\left|
		\log | p-q_j|-\frac{c_{\ell}}{2\pi} \int_{I_j}\log | p-\gamma_{\ell}(\theta)|d\theta\right|\lesssim\frac{1}{c_\ell^2}\sup_{\theta\in I_j}\left|\frac{d^2}{d\theta^2}(\log | p-\gamma_{\ell}(\theta)|)\right|.
		\]
		This second derivative has been computed in \eqref{eq:2der} and can be bounded using \eqref{eq:2der2} and \eqref{eq:normas} thus proving that
		\[
		\left|
		\log | p-q_j|-\frac{c_{\ell}}{2\pi} \int_{I_j}\log | p-\gamma_{\ell}(\theta)|d\theta\right|\lesssim\frac{1}{c_\ell^2}\sup_{\theta\in I_j}\frac{1-h_\ell^2+\sqrt{1-h_\ell^2}\sqrt{1-c^2}}{|p-\gamma_\ell(\theta)|^2}.
		\]
		But $\sqrt{1-c^2}\lesssim\sqrt{1-h_\ell^2}$ and since $j\in J$ we have $|p-\gamma_\ell (\theta)|\eqsim|p-q_j|$ for all $\theta\in I_j$, which yields
				\[
		\left|
		\log | p-q_j|-\frac{c_{\ell}}{2\pi} \int_{I_j}\log | p-\gamma_{\ell}(\theta)|d\theta\right|\lesssim\frac{1}{c_\ell^2}\frac{1-h_\ell^2}{|p-q_j|^2}.
		\]
			Recall that $\dist(p,\mathcal P_N)=| p-p_1|$ and the points $p_j$ in the parallel 
		$Q_{\ell}$ are separated by a constant times $N^{-1/2}$ and hence 
		$$|\tilde{p}-\tilde{q}_j|=|p-q_j|\gtrsim \frac{j}{\sqrt{N}},\quad 1\leq j\leq \frac{c_\ell}{2},$$
 with a similar inequality for $c_\ell/2\leq j\leq c_{\ell}$. We thus conclude that
 	\[
 \left|
 \log | p-q_j|-\frac{c_{\ell}}{2\pi} \int_{I_j}\log | p-\gamma_{\ell}(\theta)|d\theta\right|\lesssim\frac{1}{c_\ell^2}\frac{N(1-h_\ell^2)}{j^2}\eqsim\frac{1}{j^2},
 \]
 the last from Lemma \ref{lem:facil2}. We conclude that
 \[
 \eqref{eq:ultima}\lesssim\sum_{j\in J}\frac{1}{j^2}\lesssim1,
 \]
 and thus the result.
\end{proof}

\subsubsection{Parallels that are moderately close to $p$}
If $p\in B_\ell$, we will bound the terms corresponding to the parallels in $B_{\ell-1},B_{\ell}$ and $B_{\ell+1}$ (with the exception of the closest parallel to $p$, that we have already dealt with) using the following lemma. 
\begin{lemma}\label{lem:close}
	Let $p\in\S$. Then, for any $j=1,\ldots,2M-1$ such that $\dist(p,Q_{h_j})\gtrsim 1/\sqrt{N}$ then
	\[
	\left|\sum_{p_i \in Q_{h_j}} \log|p-p_i|-\left( \frac{4 \tilde{r}_j}{6}+rem_j \right)f_p(h_j)\right|\lesssim 1.
	\]
	Similarly, for any $j=1,\ldots,2M-2$ such that $\dist(p,Q_{H_j})\gtrsim 1/\sqrt{N}$ we have
	\[
	\left|\sum_{p_i \in Q_{H_j}} \log|p-p_i|-  \left( \frac{\tilde{r}_j+\tilde{r}_{j+1}}{6} \right)f_p(H_j)\right|\lesssim 1.
	\]
\end{lemma}
\begin{proof}
	We prove the first inequality since both follow from the same argument. From Lemma \ref{lem:pointstoparallels} and denoting $p=(a,b,c)$ we just need to show that $I_1+I_2+I_3\lesssim1$ where
	\begin{align*}
		I_1=&\frac{\sqrt{1-h_j^2}}{j^2}
		\int_{0}^{2\pi}\frac{1}{ |p-\gamma_{h_j}(\theta) |}\, d\theta,\\
		I_2=&\frac{\sqrt{1-c^2}(1-h_j^2)}{j^2}\int_{0}^{2\pi}\frac{1}{ |p-\gamma_{h_j}(\theta) |^3}\,d\theta,\\
		I_3=&\frac{(1-h_j^2)^{3/2}}{j^2}\int_{0}^{2\pi}\frac{1}{ |p-\gamma_{h_j}(\theta) |^3}\,d\theta.
	\end{align*}
Now, from Lemmas \ref{lem:facil2}, \ref{lem:lineintegralcuad} and \ref{lem:dist} and the hypotheses of the lemma we have
	\begin{align*}
I_1\lesssim&\frac{1}{j\sqrt{N}}\int_{0}^{2\pi}\frac{1}{ |p-\gamma_{h_j}(\theta) |}\, d\theta\lesssim\frac1j\lesssim1,\\
I_2\lesssim&\frac{\sqrt{1-c^2}}{\sqrt{N}}\int_{0}^{2\pi}\frac{1}{ |p-\gamma_{h_j}(\theta) |^2}\, d\theta\lesssim\frac{\sqrt{1-c^2}}{\sqrt{N}|h_j-c|}\lesssim1,\\
I_3\lesssim&\frac{\sqrt{1-h_j^2}}{\sqrt{N}}\int_{0}^{2\pi}\frac{1}{ |p-\gamma_{h_j}(\theta) |^2}\, d\theta\lesssim\frac{\sqrt{1-h_j^2}}{\sqrt{N}|h_j-c|}\lesssim1.
	\end{align*}
\end{proof}


\subsubsection{Parallels that are far from $p$}
Finally, assuming that $p\in B_\ell$, we bound the terms corresponding to the parallels $Q_{h_j}$ and $Q_{H_j}$ that do not touch $B_{\ell-1},B_{\ell}$ or $B_{\ell+1}$. We can therefore assume that we are under the hypotheses of Lemma \ref{lem:pointstoparallels}, that is, that for some constant $C>0$ we have
\[
\dist(p,B_j)\geq\frac{C}{\sqrt{N}}\geq \frac{C\sqrt{1-h_j^2}}{r_j}.
\]
We now prove the following result.
\begin{lemma}\label{lem:lejosmuchos}
	If $p\in B_\ell$ then
	\[
	\sum_{\stackrel{j=2}{j\neq \ell-1,\ell,\ell+1}}^{2M-2}\left[\sum_{p_i \in Q_{h_j}} \log|p-p_i|-\left( \frac{4 \tilde{r}_j}{6}+rem_j \right)f_p(h_j) \right]\lesssim1.
	\]
	Similarly,
	\[
	\sum_{\stackrel{j=2}{j\neq \ell-1,\ell,\ell+1}}^{2M-2}\left[\sum_{p_i \in Q_{H_j}} \log|p-p_i|-\left( \frac{\tilde{r}_j+\tilde{r}_{j+1}}{6} \right)f_p(H_j) \right]\lesssim1.
	\]
\end{lemma}
\begin{proof}
	We just prove the first assertion, since the second one is proved the same way.
Lemma \ref{lem:pointstoparallels} yields
\begin{multline*}
\sum_{\stackrel{j=2}{j\neq \ell-1,\ell,\ell+1}}^{2M-2}\left[\sum_{p_i \in Q_{h_j}} \log|p-p_i|-\left( \frac{4 \tilde{r}_j}{6}+rem_j \right)f_p(h_j) \right]\lesssim\\\sum_{\stackrel{j=2}{j\neq \ell-1,\ell,\ell+1}}^{2M-1}\frac{1}{j} \left( \frac{1}{\sqrt{1-h_j^2}}\int_{B_j} \frac{1}{|p-q|}d\sigma(q)+\frac{R(c,h_j)}{1-h_j^2}\int_{B_j} \frac{1}{|p-q|^3}d\sigma(q) \right).  
\end{multline*}

We split this last sum in three parts
\begin{align*}
T_1=&\sum_{j\neq \ell-1,\ell,\ell+1}
\frac{1}{j \sqrt{1-h_j^2}} \int_{B_j}\frac{1}{|p-q|}\,d\sigma(q),\\
T_2=&\sum_{j\neq \ell-1,\ell,\ell+1}
\frac{\sqrt{1-c^2}}{j}\int_{B_j}\frac{1}{|p-q|^3}\,d\sigma(q),\\
T_3=&\sum_{j\neq \ell-1,\ell,\ell+1}
\frac{\sqrt{1-h_j^2}}{j}\int_{B_j}\frac{1}{|p-q|^3}\,d\sigma(q).
\end{align*}
The easiest one is $T_3$, since from 
Lemma~\ref{lem:facil2} we have:
\begin{multline}
T_3\lesssim \frac{1}{\sqrt{N}}\sum_{j\neq \ell-1,\ell,\ell+1} 
\int_{B_j}\frac{1}{|p-q|^3}d\sigma(q)
\lesssim 
\frac{1}{\sqrt{N}}\int_{\S\setminus 
B(p,C/\sqrt{N})}\frac{1}{|p-q|^3}d\sigma(q)=\\
\frac{1}{\sqrt{N}}\int_{-1}^{1-C^2/(2N)}\frac{1}{(1-t)^{3/2}}\,dt\lesssim 1,\label{eq:pataton}
\end{multline}
where, recall, $B(p,C/\sqrt{N})$ is a spherical cap around $p$ of radius $C/\sqrt{N}$.

Now, for $T_2$, for those $j$ such that $\sqrt{1-c^2}\leq \sqrt{1-h_j^2}$ we 
apply the previous argument. In other case, again 
from 
Lemma~\ref{lem:facil2}, we have
\begin{align*}
T_2 & \lesssim 
\frac{1}{\sqrt{N}}\sum_{j\neq \ell-1,\ell,\ell+1}
\frac{\sqrt{1-c^2}}{\sqrt{1-h_j^2}}\int_{B_j}\frac{1}{|p-q |^3}\,d\sigma(q)
\\
&
=
\frac{1}{\sqrt{N}}\sum_{j\neq \ell-1,\ell,\ell+1}
\frac{|p-e_3| |p+e_3|}{|q_{h_j}-e_3| |q_{h_j}+e_3|}\int_{B_j}\frac{1}{|p-q |^3}\,d\sigma(q)
\end{align*}
where we are using that for any point $q_h \in Q_h$ we have 
$$\sqrt{1-h^2}= |q_h -e_3 | |q_h +e_3 |/2.$$  
For any $q\in B_j$ we have that
$$|q_{h_j}\pm e_3|\gtrsim |q \pm e_3|.$$
And we thus conclude 
that
\[
T_2\lesssim \frac{1}{\sqrt{N}}\sum_{j\neq \ell-1,\ell,\ell+1}
\int_{B_j}  \frac{|p-e_3| |p+e_3|}{|p-q |^3 |q -e_3| |q +e_3|} \,d\sigma(q).
\]
By a symmetry argument and using $|p-e_3|\leq  |p-q |+ |q-e_3 |$ it suffices 
to bound
\[
\frac{1}{\sqrt{N}}\int_{\{q\in\S:|p-q|\ge C/\sqrt{N}\}}\left( \frac{1}{ |p-q|^2 |q-e_3 |}+\frac{1}{|p-q |^3} \right)\,d\sigma(q),
\]
for a certain constant $C>0.$
Following the same argument as the one used for $T_3$ it is enough to consider
\begin{multline*}
\frac{1}{\sqrt{N}}\int_{\{q\in\S:|p-q|\ge C/\sqrt{N}\}} \frac{1}{ |p-q|^2 |q-e_3 |} d\sigma(q) \leq\\
\frac{1}{\sqrt{N}}\int_{\{q\in\S:|p-q|\geq C/\sqrt{N},|q-e_3|\ge 1/\sqrt{N} \}} \frac{1}{ |p-q|^2 |q-e_3 |} d\sigma(q)
\\
+
\frac{1}{\sqrt{N}}\int_{\{q\in\S:|p-q|\geq C/\sqrt{N},|q-e_3|\le 1/\sqrt{N} \}} \frac{1}{ |p-q|^2 |q-e_3 |} d\sigma(q).
\end{multline*}
The first of these two integrals is from H\"older's inequality at most
\[
\left(\int_{\{q\in\S:|p-q|\ge  C/\sqrt{N} \}}\frac{1}{ |p-q |^3}\,d\sigma(q)\right)^{2/3}
\left(\int_{
\{q\in\S: |q-e_3|\ge 1/\sqrt{N} \} }\frac{1}{ |q-e_3 |^3}\,
d\sigma(q)\right)^{1/3},
\]
which is bounded above again by a constant times $\sqrt{N}$ as already seen in \ref{eq:pataton}. We bound the second integral as
\[
N \int_{\{ q\in\S:|q-e_3|\le 1/{\sqrt{N}}\} }\frac{1}{ |q-e_3 |}\,
d\sigma(q)
\lesssim  N\int_{1-\frac{2}{N}}^1\frac{1}{\sqrt{1-t}}\,dt\lesssim \sqrt{N}.
\]

It remains to bound $T_1$. Again from Lemma~\ref{lem:facil2} we have
\begin{align*}
T_1 \lesssim \frac{1}{\sqrt{N}}
\sum_{j\neq \ell-1,\ell,\ell+1}
\frac{1}{1-h_j^2} & \int_{ B_j}\frac{1}{ |p-q |}\,d\sigma(q)
\lesssim \sqrt{N}\int_{B_1\cup B_{2M-1}} \frac{1}{ |p-q |}\,d\sigma(q)
\\
 +
 &
 \frac{1}{\sqrt{N}}\left[ \sum_{j=2}^{\ell-2}+\sum_{j=\ell+2}^{2M-2} \right]
\frac{1}{1-h_j^2}\int_{ B_j}\frac{1}{ |p-q |}\,d\sigma(q),
\end{align*}
where we have used that $1-h_j^2\gtrsim 1/N$. The first integral is easily bounded by
\[
\int_{q\in B_1 \cup B_{2M-1}}\frac{1}{|p-q|}\,d\sigma(q)\leq
\int_{ B_1}\frac{1}{ |e_3-q |}\,d\sigma(q)+\int_{ B_{2M-1}}\frac{1}{ |e_3+q |}\,d\sigma(q)\lesssim \frac{1}{\sqrt{N}}.
\]
Finally, from the same arguments as above we have to bound
\[
\frac{1}{\sqrt{N}} \int_{\{q\in \S: |p-q |\ge C/\sqrt{N}, |q\pm e_3 |\geq C'/\sqrt{N}\}}  \frac{1}{ |p-q | |q-e_3 |^2 |q+e_3 |^2}\,d\sigma(q),
\]
where $C,C'$ are positive constants and it is enough to check that 
\[
\int_{\{q\in \S: |p-q |\ge C/\sqrt{N}, |q-e_3 |\geq 
C'/\sqrt{N} \}}\frac{1}{|p-q||q-e_3|^2}\,d\sigma(q),
\]
\[
\int_{\{q\in \S:|p-q|\ge C/\sqrt{N},|q+e_3|\geq 
C'/\sqrt{N}\}}\frac{1}{|p-q||q+e_3|^2}\,d\sigma(q),
\]
are $O(\sqrt{N}).$ This again follows from H\"older's inequality.
\end{proof}


\subsection{Proof of Theorem \ref{th:multiplier}}
We are now ready to finish the proof. By Proposition \ref{bandparallel}
\begin{align*}
\sum_{i=1}^N & \log | p-p_i|=
\sum_{\stackrel{i=1}{i\neq i_0}}^N \log | p-p_i| +\log \dist(p,\mathcal P_N) = -\kappa 
N+\log \left(\sqrt{N} \dist(p,\mathcal P_N)\right)
\\
&
+\sum_{p_i \not\in Q_{\ell}}\log | p-p_i |-S_N(\ell) 
+
\sum_{\stackrel{p_i \in Q_{\ell}}{i\neq i_0}}\log | 
p-p_i|-c_\ell f_p(h_{\ell})-\frac12\log N+O(1),
\end{align*}
where $S_N(\ell)$ is the sum $S_N$ without the part corresponding to the parallel $Q_{\ell}$ and $c_\ell\eqsim \ell$ is the number of points in parallel $Q_\ell$. From Lemmas \ref{lem:veryclose}, \ref{lem:close} and \ref{lem:lejosmuchos} we conclude that
\[
\left|\sum_{i=1}^N \log  |p-p_i |+ \kappa N - \log \left(\sqrt{N} \dist(p,\mathcal P_N)\right)\right|\lesssim1,
\]
as wanted.


\appendix
\section{The error of the mid-point rule for numerical integration}
Recall the following classical estimates for the midpoint and Simpson integration rules.
\begin{lemma}\label{lem:midpointbasico}
	Let $f:[a,b]\to\R$ be a $C^2$ function. Then,
	\[
	\left|\int_a^bf(x)\,dx-(b-a)f\left(\frac{a+b}{2}\right)\right|\leq\frac{(b-a)^3 \|f^{(2)}\|_\infty}{24}.
	\]
	Moreover, if $f$ is $C^4$ then,
	\[
	\left|\int_a^bf(x)\,dx-\frac{b-a}{6}\left(f(a)+4f\left(\frac{a+b}{2}\right)+f(b)\right)\right|\leq\frac{(b-a)^5\|f^{(4)}\|_\infty}{2880}.
	\]
\end{lemma}
We also need the following more sophisticated version of the midpoint rule.
\begin{lemma}\label{lem:midpoint}
Let $f:[a,b]\to\R$ be a $C^3$ function. Then,
\[
\left|\int_a^bf(x)\,dx-(b-a)f\left(\frac{a+b}{2}\right)-\frac{(b-a)^2}{24}
\int_a^bf''(x)\,dx\right|\leq\frac{(b-a)^4\|f^{(3)}\|_\infty}{64}.
\]
\end{lemma}
\begin{proof}
We first assume that $[a,b]=[-1,1]$. Let $S$ be the quantity to be estimated in 
this lemma. Expanding with Taylor series
\[
f(x)=f(0)+f'(0)x+\frac12 f''(0)x^2+\frac16 f^{(3)}(\zeta_x)x^3.
\]
\[
f''(x)=f''(0)+f^{(3)}(\eta_x)x,
\]
then the quantity to be estimated is
\[
\left|\int_{-1}^1\frac12 
f''(0)x^2+\frac{1}{6}f^{(3)}(\zeta_x)x^3\,dx-\frac{2f''(0)}{6}-\frac{1}{6}\int_{-1
}^1f^{(3)}(\eta_x)x\,dx\right|\le \frac{\|f^{(3)}\|_\infty}{4}.
\]
For general $[a,b]$ one can apply the previous result to $g:[-1,1]\to\R$ 
given by $g(t)=f((a+b)/2+t(b-a)/2)$.
\end{proof}


\end{document}